\documentclass[12pt]{article}
\usepackage{amsthm,amstext, amsmath,latexsym,amsbsy,amssymb,cases}
\usepackage[final]{pdfpages}

\newtheorem{cor}{Corollary}[section]

\newtheorem{lem}[cor]{Lemma}
\newtheorem{prop}[cor]{Proposition}

\newtheorem{propo}{Proposition}
\newtheorem{theor}[propo]{Theorem}

\title{\bf Blow-up rate for a semilinear wave equation with exponential nonlinearity in one space dimension}
 \author{\small Asma AZAIEZ\footnote{This author is supported by the ERC Advanced Grant no. 291214, BLOWDISOL.}  \\
\small Universit\'e de Cergy-Pontoise, \\
\small AGM, CNRS (UMR 8088), 95302, Cergy-Pontoise, France.\\
 \small Nader MASMOUDI\footnote{This author is partially supported by NSF grant DMS-
1211806.}\\
\small Courant Institute of Mathematical Sciences.\\
\small Hatem ZAAG\footnote{This author is supported by the ERC Advanced Grant no. 291214, BLOWDISOL. and by ANR project ANA\'E ref. ANR-13-BS01-0010-03.} \\
\small Universit\'e Paris 13, Sorbonne Paris Cit\'e \\
\small LAGA, CNRS (UMR 7539), F-93430, Villetaneuse, France.}

\begin{document}
\maketitle
\begin{abstract} 
 We consider in this paper blow-up solutions of the semilinear wave equation in one space dimension, with an exponential source term. Assuming that initial data are in $H^{1}_{loc}\times L^2_{loc}$ or some times in $ W^{1,\infty}\times L^{\infty}$, we derive the blow-up rate near a non-characteristic point in the smaller space,
  and give some bounds near other points. Our result generalize those proved by Godin under high regularity assumptions on initial data.

 \end{abstract}

\section {Introduction}
We consider the one dimensional semilinear wave equation:
\begin{equation} \left\{
\begin{array}{l}
\displaystyle\partial^2_{t} u = \partial^2_{x} u+e^{u}, \\
u(0)=u_{0} \mbox{ and }  \partial_{t}u(0) = u_{1}
\end{array}
\right . \label{waveq}
\end{equation}
where $u(t):  x\in \mathbb R\to u(x,t) \in \mathbb R , u_0 \in
H^{1}_{loc,u}$ and $ u_1\in L^2_{loc,u}$. We may also add more restriction on initial data by assuming that $(u_0,u_1)\in W^{1,\infty}\times L^{\infty}.$
 The Cauchy problem for equation (\ref{waveq}) in the space $H^{1}_{loc,u}\times L^2_{loc,u}$ follows from fixed point techniques (see Section \ref{cauchy-Problem} below).
\\If the solution is not global in time, we show in this paper that it blows up (see Theorems \ref{th} and \ref{NEW} below). For that reason, we call it a blow-up
solution. The existence of blow-up solutions is guaranteed by ODE techniques and the finite speed of propagation.

\medskip

More blow-up results can be found in Kichenassamy and Littman
 \cite{KL93b}, \cite{KL93a}, where the authors introduce a systematic procedure for
reducing nonlinear wave equations to characteristic problems of
Fuchsian type and construct singular solutions of general semilinear
equations which blow up on a non characteristic surface, provided
that the first term of an expansion of such solutions can be found.

The case of the power nonlinearity has been understood completely in a series of papers, in the real case (in one space dimension) by Merle and Zaag \cite{MR2362418}, \cite{MR2415473}, \cite{MR2931219} and \cite{Mz12} and in C\^ote and Zaag \cite{CZ12} (see also the note \cite{Mz10}), in the complex case by Azaiez \cite{<3}. Some of those results have been extended to higher dimensions for conformal or subconformal $p$:
\begin{equation}\label{conp}
 1<p\le p_c \equiv 1+\frac{4}{N-1},
\end{equation}
under radial symmetry outside the origin in \cite{MR2799813}. For non radial solutions, we would like to mention \cite{MZ05} and \cite{MZ2005} where the blow-up rate was obtained. We also mention the recent contribution of \cite{MZ13} and \cite{MZ13'} where the blow-up behavior is given, together with some stability results.

In \cite{MR849476} and \cite{MR802832}, Caffarelli and Friedman considered semilinear wave equations with a nonlinearity of power type. If the space dimension $N$ is at most $3$, they showed in \cite{MR849476} the existence of solutions of Cauchy problems which blow up on a $C^1$ spacelike hypersurface. If $N = 1$ and under suitable assumptions, they obtained in \cite{MR802832} a very general result which shows that solutions of Cauchy problems either are global or blow up on a $C^1$ spacelike curve. In \cite{MR1854257} and \cite{MR1804655}, Godin shows that the solutions of Cauchy problems either are global or blow up on a $C^1$ spacelike curve for the following mixed problem ($\gamma\neq 1$, $|\gamma|\ge 1$)
\begin{equation} \left\{
\begin{array}{l}
\displaystyle\partial^2_{t} u = \partial^2_{x} u+e^{u},\,x>0 \\
\partial_x u+\gamma\partial_t u=0 \mbox{ if }x=0.
\end{array}
\right .
\end{equation}
In \cite{MR1854257}, Godin gives sharp upper and lower bounds on the blow-up rate for initial data in $C^4\times C^3 $. It happens that his proof can be extended for initial data $(u_0,u_1)\in H^{1}_{loc,u}\times L^2_{loc,u}$ (see Proposition \ref{p} below).

\bigskip

Let us consider \emph{u} a blow-up solution of (\ref{waveq}). Our aim in this paper is to derive upper and lower estimates on the blow-up rate of $u(x,t)$. In particular, we first give general results (see Theorem \ref{th} below), then, considering only non-characteristic points, we give better estimates in Theorem \ref{NEW}.

\medskip
From Alinhac \cite{ali95}, we define a continuous curve $\Gamma$ as the graph of a function ${x \mapsto T(x)}$ such that the domain of definition of $u$ (or the maximal influence domain of $u$) is 
\begin{equation}\label{domaine-de-definition}
D=\{(x,t)|\, 0\le t<T(x)\}.
\end{equation}
From the finite speed of propagation, $T$ is a 1-Lipschitz function. The graph $\Gamma$ is called the blow-up graph of $u$.

Let us introduce the following non-degeneracy condition for $\Gamma$. If we introduce for all $x \in \mathbb{R},$ $t\le T(x)$ and $\delta>0$, the cone
\begin{equation}\label{3,5}
 \mathcal{C}_{x,t, \delta }=\{(\xi,\tau)\neq (x,t)\,| 0\le \tau\le t-\delta |\xi-x|\},
\end{equation}
then our non-degeneracy condition is the following: $x_0$ is a non-characteristic point if
\begin{equation}\label{4}
 \exists \delta_0 = \delta_0 (x_0) \in (0,1) \mbox{ such that } u \mbox{ is defined on }\mathcal{C}_{x_0,T(x_0), \delta_0}.
\end{equation}
If condition (\ref{4}) is not true, then we call $x_0$ a characteristic point. We denote by $\cal{ R} \subset \mathbb{ R}$ (resp. $\cal{ S} \subset \mathbb{ R}$) the set of non-characteristic (resp. characteristic) points.\\

\medskip
We also introduce for each $a\in \mathbb R $ and $T\leq T(a)$
 the following similarity variables:
\begin{equation}\label{trans_auto}
w_{a,T}(y,s) =u(x,t)+2\log(T-t), \quad  y=\frac{x-a}{T-t}, \quad
s=-\log(T-t).\end{equation}
\\If $T=T(a)$, we write $w_a$ instead of $w_{a,T(a)}$.
\\From equation (\ref{waveq}), we see that $w_{a,T}$ (or $w$ for simplicity) satisfies, for all $s
\geq -\log T $, and $y \in (-1,1)$,
\begin{align}\label{equaw}
 {\partial^2_{s} w}-{\partial_y} ((1-y^2){\partial_{y} w})-e^w+2=-  {\partial_s w}-2 y  \partial^2_{y,s}w.
\end{align}
\\In the new set of variables $(y,s)$, deriving the behavior of $u$ as
 $t\to T$ is equivalent to studying the behavior of \emph{w} as s $\to
+\infty$.\\

Our first result gives rough blow-up estimates. Introducing the following set
\begin{align}\label{6.5}
 D_R\equiv\{(x,t)\in (\mathbb R,\mathbb R_+),|x|<R-t\},
\end{align}
where $R>0$, we have the following result
\begin{theor}{\bf (Blow-up estimates near any point)}\label{th} We claim the following:

 \begin{itemize}
\renewcommand{\labelitemi}{}
 \item{i) }\emph{\bf(Upper bound)} For all $R>0$ and $a\in\mathbb{R}$ such that $(a,T(a))\in D_R$, it holds:

\begin{eqnarray*}
\forall |y|<1, \,\forall s\ge -\log T(a),\,w_a(y,s)\le -2 \log(1-|y|)+C(R),
\end{eqnarray*}
\begin{eqnarray*}
\forall t\in[0,T(a)),\,e^{u(a,t)}\le\frac{C(R)}{d((a,t),\Gamma)^2}\le\frac{C(R)}{(T(a)-t)^2}.
\end{eqnarray*}
where $d((x,t),\Gamma)$ is the (Euclidean) distance from $(x,t)$ to $\Gamma$. 
\item{ii) }\emph{\bf(Lower bound)}
For all $R>0$ and $a\in \mathbb{R}$ such that $(a,T(a))\in D_R$, it holds that

$$\frac{1}{T(a)-t}\int_{I(a,t)} e^{-u(x,t)}dx \le C(R)\sqrt{d((a,t),\Gamma)}\le C(R)\sqrt{T(a)-t}.$$

If in addition, $(u_0,u_1)\in W^{1,\infty}\times L^{\infty}$ then
$$\forall t \in [0,T(a)),e^{u(a,t)}\ge \frac{C(R)}{d((a,t),\Gamma)}\ge \frac{C(R)}{T(a)-t}.$$
\item{iii) }\emph{\bf(Lower bound on the local energy "norm")}
There exists $ \epsilon_0>0$ such that for all 
$ a \in \mathbb{R}$, and $t \in [0, T(a)),$
\begin{eqnarray}\label{10,1}
\frac{1}{T(a)-t} &\displaystyle \int_{I(a,t)} ((u_t(x,t))^2+(u_x(x,t))^2+e^{u(x,t)})dx \geq
&\dfrac{\epsilon_0}{(T(a)-t)^2},
\end{eqnarray}
where $I(a,t)=(a-(T(a)-t),a+(T(a)-t)).$
 \end{itemize}
\end{theor}
\noindent {\bf Remark}: The upper bound in item $i)$ was already proved by Godin \cite{MR1854257}, for more regular initial data. Here, we show that Godin's strategy works even for less regular data. We refer to the integral in (\ref{10,1}) as the local energy "norm", since it is like the local energy as in Shatah and Struwe, though with the "$+$" sign  in front of the nonlinear term. Note that the lower bound in item $iii)$ is given by the solution of the associated ODE $u''=e^u$. However the lower bound in $ii)$ doesn't seem to be optimal, since it does not obey the ODE behavior.  Indeed, we expect the blow-up for equation (\ref{waveq}) in the "ODE style", in the sense that the solution is comparable to the solution of the ODE $u"=e^u$ at blow-up. This is in fact the case with regular data, as shown by Godin  \cite{MR1854257}. \\

\bigskip

\noindent If in addition $a \in \cal{R}$, we have optimal blow-up estimates:
\begin{theor}{ \bf (An optimal bound on the blow-up rate near a non-characteristic point in a smaller space)}\label{NEW} 
Assume that $(u_0,u_1)\in W^{1,\infty}\times L^{\infty}$. Then, for all $R>0$, for any $a \in \cal{R}$ such that $(a,T(a))\in D_R$, we have the following
\begin{itemize}
\renewcommand{\labelitemi}{}
\item{i) }\emph{\bf(Uniform bounds on $w$)} For all $s\geq -\log T(a)+1,$
\begin{eqnarray*}
| w_a(y,s)|+\int_{-1}^{1}
\left( ({\partial_s w_a}(y,s))^2 +
({\partial_y w_a}(y,s))^2\right)\, dy\le C(R)
\end{eqnarray*}
where $w_a$ is defined in (\ref{trans_auto}).
 \noindent \item{ii) }\emph{\bf(Uniform bounds on $u$)} For all $t \in[0,T(a))$
 \begin{eqnarray*}
| u(x,t)+2\log(T(a)-t)|+(T(a)-t)\int_{I}
 (\partial_x u (x,t))^2 + (\partial_t u(x,t))^2\, dx\le C(R).
 \end{eqnarray*}
 In particular, we have
  \begin{eqnarray*}
 \frac{1}{C(R)}\le e^ {u (x,t) }(T(a)-t)^2\, \le C(R).
 \end{eqnarray*}
 \end{itemize}
\end{theor}

\noindent {\bf Remark}: This result implies that the solution indeed blows up on the curve $\Gamma$.

\noindent {\bf Remark}: Note that when $a\in  \cal{R}$, Theorem \ref{th} already holds and directly follows from Theorem \ref{NEW}. Accordingly, Theorem \ref{th} is completely meaningful when $a\in \cal{S}$.\\
Following Antonini, Merle and Zaag in \cite{AM01} and \cite{MZ2005}, we would like to mention the existence of a Lyapunov functional in similarity variables. More precisely, let us define 
\begin{equation}\label{lyapunov} E(w(s))=\int_{-1}^{1}
\left(\frac{1}{2} ({\partial_s w})^2 +\frac{1}{2} (1-y^2)
({\partial_y w})^2-e^w+2 w\right)\, dy.
\end{equation}
We claim that the functional $E$ defined by (\ref{lyapunov}) is a
decreasing function of time for solutions of (\ref {equaw}) on
(-1,1).

\begin{propo}{\bf (A Lyapunov functional for equation (\ref{waveq}))}\label{2.1}
For all $a\in \mathbb{ R},\, T \le T(a)
 ,\,
s_2\geq s_1\geq-\log T$, the following identities hold for $w=w_{a,T}$:
\begin{eqnarray*}
E(w(s_2))- E(w(s_1))=-\int_ {s_1}^{s_2} ({\partial_s w(-1,s)})^2+(
{\partial_s w(1,s)})^2 ds.\end{eqnarray*}

\end{propo}

\noindent{\bf Remark}: The existence of such an energy in the
context of the nonlinear heat equation has been introduced by Giga
and Kohn in \cite{gk85}, \cite{gk87} and \cite{gk89}.

\noindent{\bf Remark}: As for the semilinear wave equation with conformal power nonlinearity, the dissipation of the energy $E(w)$ degenerates to the boundary $\pm 1$.

\bigskip
This paper is organized as follows:\\
In Section \ref{cauchy-Problem}, we solve the local in time Cauchy problem.\\
Section \ref{sec2} is devoted to some energy estimates.\\
In Section \ref{sec3}, we give and prove upper and lower bounds, following the strategy of Godin \cite{MR1854257}.\\
Finally, Section \ref{sec4} is devoted to the proofs of Theorem \ref{th}, Theorem \ref{NEW} and Proposition \ref{2.1}.
\section{The local Cauchy problem} \label{cauchy-Problem}
In this section, we solve the local Cauchy problem associated to (\ref{waveq}) in the space $H^{1}_{loc,u}\times L^2_{loc,u}$. In order to do so, we will proceed in two steps :\\
\noindent $1)$ In Step 1, we solve the problem in $H^{1}_{loc,u}\times L^2_{loc,u}$, for some uniform $T>0$ small enough. \\
\noindent $2)$ In Step 2, we consider $x_0\in \mathbb{R}$, and use step 1 and a truncation to find a local solution defined in some cone $\mathcal{C}_{x_0,\tilde T(x_0),1}$ for some $\tilde T(x_0)>0$. Then, by a covering argument, the maximal domain of definition is given by $D=\displaystyle\cup_{x_0\in \mathbb{R}}\mathcal{C}_{x_0,\tilde T(x_0),1}$.\\
\noindent $3)$ In Step 3, we consider some approximation of equation (\ref{waveq}), and discuss the convergence of the approximating sequence.
\medskip

\noindent {\bf Step 1: The Cauchy problem in $H^{1}_{loc,u}\times L^2_{loc,u}$ }\\
In this step, we will solve the local Cauchy problem associated to (\ref{waveq}) in the space $H=H^{1}_{loc,u}\times L^2_{loc,u}$. In order to do so, we will apply a fixed point technique. We first introduce the wave group in one space dimension:
\begin{eqnarray*}
S(t):H&\rightarrow& H\\
(u_0, u_1)&\mapsto&S (t)( u_0,  u_1)(x),
\end{eqnarray*}

\begin{eqnarray*}
S (t)( u_0,  u_1)(x)=
\begin{pmatrix}\displaystyle
\frac{1}{2}(u_0(x+t)+ u_0(x-t))+\frac{1}{2} \int_{x-t}^{x+t}u_1\\
\frac{1}{2}(u_0'(x+t)- u_0'(x-t))+\frac{1}{2}(u_1(x+t)+ u_1(x-t))
 \end{pmatrix}.
\end{eqnarray*}

Clearly, $S(t)$ is well defined in $H$, for all $t\in \mathbb{R}$, and more precisely, there is a universal constant $C_{0}$ such that
 \begin{eqnarray}\label{continS}
 ||S(t)(u_0, u_1)||_{H}\leq C_{0} (1+t) ||(u_0, u_1)||_{H}.
 \end{eqnarray}

This is the aim of the step:
\begin{lem}{\bf (Cauchy problem in $H^{1}_{loc,u}\times L^2_{loc,u}$)}\label{a1}
 For all $(u_0,u_1)\in H$, there exists $T>0$ such that there exists a unique solution of the problem (\ref{waveq}) in $C([0,T],H).$ 
\end{lem}
\begin{proof}
Consider $T>0$ to be chosen later small enough in terms of $ ||(u_0, u_1)||_{H}$.

\noindent We first write the Duhamel formulation for our equation
\begin{eqnarray}\label{DUHAMEL}
u(t)=S(t)(u_0, u_1)+\int_{0}^{t} S(t-\tau)(0, e^{u(\tau)}) d\tau.
\end{eqnarray}
Introducing
\begin{eqnarray}\label{rayon}
R=2 C_0 (1+T) ||(u_0, u_1)||_{H},
\end{eqnarray} we will work in the Banach space $E=C([0, T], H)$ equipped with the norm $||u||_{E}=\sup\limits_{0\leq t \leq T
}||u ||_{H}.$
Then, we introduce
\begin{eqnarray*}
\Phi:E&\rightarrow& E\\
V(t)=\begin{pmatrix}v(t)\\ v_{1}(t)\end{pmatrix}&\mapsto& S (t)(u_0, u_1)+\int_{0}^{t}S(t-\tau)(0, e^{v(t)}) d\tau \displaystyle
\end{eqnarray*}
and the ball $B_E(0, R)$.\\
\medskip
We will show that for $T>0$ small enough, $\Phi$ has a unique fixed point in $B_E(0, R)$. To do so, we have to check 2 points:

\begin{enumerate}
\item $\Phi$ maps $B_E(0, R)$ to itself.
\item $\Phi$ is k-Lipschitz with $k<1$ for $T$ small enough.
\end{enumerate}
\begin{itemize}
\renewcommand{\labelitemi}{$\bullet$}
\item Proof of 1: Let $V=\begin{pmatrix}v\\ v_{1}\end{pmatrix}\in B_E(0, R)$, this means that: $$\forall t \in [0, T], \; v(t) \in H_{loc, u}^{1}(\Bbb R)\subset L^{\infty}(\Bbb R)$$ and that $$||v(t)||_{L^{\infty}(\Bbb R)}\leq C_{*}R.$$\\
Therefore
\begin{eqnarray}\label{carreaux}
||(0, e^{v})||_{E}&=&\sup\limits_{0\leq t \leq T}||e^{v(t)} ||_{L^2_{loc, u}}\notag\\
&\leq& e^{C_{*}R\sqrt{2}}.
\end{eqnarray}
This means that $$\forall \tau \in [0, T]\; (0, e^{v(\tau)}) \in H,$$ hence $S(t-\tau)(0, e^{v(\tau)})$ is well defined from (\ref{continS}) and so is its integral  between $0$ and $t$.
So $\Phi$ is well defined from $E$ to $E$.

\bigskip \noindent Let us compute $||\Phi(v)||_{E}$:\\ Using (\ref{continS}), (\ref{rayon}) and (\ref{carreaux}) we write for all $t\in [0, T]$,
\begin{eqnarray}
||\Phi(v)(t)||_{H}&\leq &|| S (t)(u_0, u_1)||_{H}+\int_{0}^{t}|| S (t-\tau)(0, e^{v(\tau)})||_{H}d\tau \notag\\
&\leq & \frac{R}{2}+\int_{0}^{T} C_{0}(1+T) \sqrt{2} e^{C_* R} d\tau \notag\\
&\leq & \frac{R}{2}+C_{0}T(1+T)\sqrt{2}e^{C_* R}.
\end{eqnarray}
Choosing $T$ small enough so that $$\displaystyle \frac{R}{2}+C_{0}T(1+T)\sqrt{2}e^{C_* R}\leq R$$
or $$T(1+T)\leq \frac{R e^{-C_* R}}{2\sqrt{2}C_0}$$ guarantees that $\Phi$ goes from $B_E(0, R)$ to $B_E(0, R)$.
\bigskip
\item Proof of 2: Let $V, \bar{V} \in B_E(0, R)$ we have
$$\Phi(V)-\Phi(\bar{V})=\int_{0}^{T}  S (t-\tau)(0, e^{v(t)}-e^{\bar{v}(t)})d\tau.$$
Since $ ||v(t)||_{L^{\infty}(\Bbb R)}\leq C_* R$ and the same for $ ||\bar{v}(t)||_{L^{\infty}(\Bbb R)}$, we write
$$|e^{v(\tau)}-e^{\bar{v}(\tau)}|\leq e^{C_* R} |v(\tau)-\bar{v}(\tau)|, $$

hence
\begin{eqnarray}
||e^{v(\tau)}-e^{\bar{v}(\tau)}||_{L^2_{loc, u}}&\leq& e^{C_* R} ||v(\tau)-\bar{v}(\tau)||_{L^2_{loc, u}}\notag \\
&\leq& e^{C_* R} ||V-\bar{V}||_{E}.
\end{eqnarray}
Applying $S(t-\tau)$ we write from (\ref{continS}), for all $0\le \tau \le t\le \tau \in  T,$
\medskip
\begin{eqnarray}\displaystyle
||S(t-\tau) (0, e^{v(\tau)}-e^{\bar{v}(\tau)})||_{H}&\leq& C_0 (1+T)||(0, e^{v(\tau)}-e^{\bar{v}(\tau)})||_{H}\notag \\
&\leq& C_0 (1+T)|| e^{v(\tau)}-e^{\bar{v}(\tau)}||_{L^2_{loc, u}}\notag \\
&\leq&C_0 (1+T) e^{C_* R} ||V-\bar{V}||_{E}.
\end{eqnarray}
Integrating we end-up with
\medskip
\begin{eqnarray}
||\Phi(V)-\Phi(\bar{V})||_{E}\leq C_0 T(1+T) e^{C_* R} ||V-\bar{V}||_{E}.
\end{eqnarray}
\medskip
$k= C_0 T(1+T) e^{C_* R}$ can be made $<1$ if $T$ is small.\\

\end{itemize}
{\it Conclusion}: From points 1 and 2, $\Phi$ has a unique fixed point $u(t)$ in $B_E (0, R)$. This fixed point is the solution of the Duhamel formulation (\ref{DUHAMEL}) and of our equation (\ref{waveq}). This concludes the proof of Lemma \ref{a1}.
\end{proof}
\noindent {\bf Step 2: The Cauchy problem in a larger region}\\

\noindent Let $( u_0,  u_1)\in H^{1}_{loc,u}\times L^2_{loc,u}$ initial data for the problem (\ref{waveq}). Using the finite speed of propagation, we will localize the problem and reduces it to the case of initial data in $H^{1}_{loc,u}\times L^2_{loc,u}$ already treated  in Step 1. For $(x_0,t_0)\in \mathbb{R}\times (0,+\infty)$, we will check the existence of the solution in the cone $\mathcal{C}_{x_0,t_0,1}$. In order to do so, we introduce $\chi $ a $C^\infty$ function with compact support such that $ \chi(x)= 1$ if $|x-x_0|<t_0$, let also $(\bar u_0, \bar u_1)=(u_0\chi ,u_1\chi )$ (note that $\bar u_0$ and $\bar u_1$ depend on $(x_0,t_0)$ but we omit this dependence in the indices for simplicity). So, $(\bar u_0, \bar u_1)\in  H^{1}_{loc,u}\times L^2_{loc,u}$. From Step 1, if $\bar u$ is the corresponding solution of equation (\ref{waveq}), then, by the finite speed of propagation, $u=\bar u$ in the intersection of their domains of definition with the cone $\mathcal{C}_{x_0,t_0,1}$. As $\bar u$ is defined for all $(x,t)$ in $\mathbb{R}\times [0,T)$ from Step 1 for some $T=T(x_0,t_0)$, we get the existence of $u$ locally in $\mathcal{C}_{x_0,t_0,1}\cap\mathbb{R}\times [0,T)$.
Varying $(x_0,t_0)$ and covering $\mathbb{R}\times (0,+\infty[$ by an infinite number of cones, we prove the existence and the uniqueness of the solution in a union of backward light cones, which is either the whole half-space $\mathbb{R}\times (0,+\infty)$, or the subgraph of a 1-Lipschitz function $x\mapsto T(x)$. We have just proved the following:
\begin{lem}{\bf (The Cauchy problem in a larger region)}
Consider $( u_0,  u_1)\in H^{1}_{loc,u}\times L^2_{loc,u}$. Then, there exists a unique solution defined in $D$, a subdomain of $\mathbb{R}\times [0,+\infty)$, such that for any $(x_0,t_0)\in D, (u,\partial_tu)_{(t_0)}\in H^1_{loc}\times L^2_{loc}(D_{t_0})$, with $D_{t_0}=\{x\in\mathbb{R}|(x,t_0)\in D\}$. Moreover,
\begin{itemize}
 \item [$-$] either $D=\mathbb{R}\times [0,+\infty)$,
\item [$-$] or $D=\{(x,t)| 0\le t< T(x)\}$ for some 1-Lipschitz function $x\mapsto T(x)$.
\end{itemize}

\end{lem}

\noindent {\bf Step 3: Regular approximations for equation (\ref{waveq})}\\

\noindent Consider $( u_0, u_1)\in H^{1}_{loc,u}\times L^2_{loc,u}$, $u$ its solution constructed in Step 2, and assume that it is non global,
hence defined under the graph of a 1-Lipschitz function $x\mapsto T(x).$ Consider for any $n\in \mathbb{N}$, a regularized increasing truncation of $F$ satisfying 

\begin{equation} F_n(u)=\left\{
\begin{array}{l}
\displaystyle e^{u}\mbox{ if }  u\le n \\
e^{n}\mbox{ if }  u\ge n+1
\end{array}
\right .\label{16.3}
\end{equation}
and $F_n(u)\le \min(e^u,e^{n+1})$. Consider also a sequence $(u_{0,n},u_{1,n})\in (C^\infty(\mathbb{R}))^2$ such that $(u_{0,n},u_{1,n})\rightarrow (u_0,u_1)$ in $H^1\times L^2(-R,R)$ as $n\rightarrow \infty$, for any $R>0$.\\
Then, we consider the problem
\begin{equation} \left\{
\begin{array}{l}
\displaystyle\partial^2_{t} u_n = \partial^2_{x}  u_n + F_n(u_n), \\
(u_n(0),\partial_t u_{n}(0)) =(u_{0,n},u_{1,n})\in H^{1}_{loc,u}\times L^2_{loc,u}.
\end{array}
\right . \label{16.5}
\end{equation}
Since Steps 1 and 2 clearly extend to locally Lipschitz nonlinearities, we get a unique solution $u_n$ defined in the half-space  $\mathbb{R}\times (0,+\infty)$, or in the subgraph of a 1-Lipschitz function.
Since $F_n(u)\le e^{n+1}$, for all $u\in \mathbb{R}$, it is easy to see that in fact: $u_n$ is defined for all $(x,t)\in \mathbb{R}\times [0,+\infty)$. 
From the regularity of $F_n$, $u_{0,n}$ and $u_{1,n}$, it is clear that $u_n$ is a strong solution in $C^2(\mathbb{R},[0,\infty))$.
Introducing the following sets:
\begin{align}\label{16.55}
 K^+(x,t)=\{(y,s)\in (\mathbb R,\mathbb R_+),|y-x|<s-t\},
\end{align}
$$K^-(x,t)=\{(y,s)\in (\mathbb R,\mathbb R_+),|y-x|<t-s\},$$
 and 
$$K^\pm_R(x,t)=K^\pm(x,t)\cap \overline{ D_R}.$$
we claim the following
\begin{lem}{\bf (Uniform bounds on variations of $u_n$ in cones)}\label{lab}
 Consider $R>0$, one can find $C(R)>0$ such that if $(x,t)\in D\cap \overline{D_R},$ then $\forall n\in \mathbb{N}$:
$$ u_n(y,s)\ge u_n(x,t)-C(R),\; \forall (y,s)\in \overline{K^+_R(x,t)},$$
$$u_n(y,s)\le u_n(x,t)+C(R),\; \forall (y,s)\in \overline{K^-(x,t)}.$$
\end{lem}
\noindent {\bf Remark: } Of course $C$ depends also on initial data, but we omit that dependence, since we never change initial data in this setting. Note that since $(x,t)\in \overline{ D_R}$, it follows that $K^-_{ R}(x,t)=K^-(x,t)$.
\begin{proof} We will prove the first inequality, the second one can be proved by the same way. For more details see \cite{MR1854257} page $74$.\\
Let $R>0$, consider $(x,t)$ fixed in $D\cap\overline{D_R},$ and $(y,s)$ in $D\cap \overline{K^+_R(x,t)}$. We introduce the following change of variables: 
\begin{eqnarray}\label{21,5}
 \xi=(y-x)-(s-t),\,\eta=-(y-x)-(s-t),\,\bar u_n(\xi,\eta)=u_n(y,s).
\end{eqnarray}

From (\ref{16.5}), we see that $\bar u_n$ satisfies :
\begin{eqnarray}\label{xieta}
 \partial_{\xi \eta } \bar u_n(\xi,\eta)=\frac{1}{4}F_n(\bar u_n)\ge 0.
\end{eqnarray}
Let $(\bar \xi,\bar \eta)$ the new coordinates of $(y,s)$ in the new set of variables. Note that $\bar\xi\le0$ and $\bar\eta\le0$. We note that there exists $\xi_0\ge0$ and $\eta_0\ge0$ such that the points $(\xi_0,\bar \eta)$ and $(\bar \xi, \eta_0)$ lay on the horizontal line $\{s=0\}$ and have as original coordinates respectively $(y^*,0)$ and $(\tilde y,0)$ for some $y^*$ and $\tilde y$ in $[-R,R]$. We note also that in the new set of variables, we have :

\begin{eqnarray}\label{equadi1}
&u_n(y,s)-u_n(x,t)=\bar u_n(\bar \xi,\bar \eta)-\bar u_n(0,0)=\bar u_n(\bar \xi,\bar \eta)-\bar u_n(\bar \xi,0)+\bar u_n(\bar \xi,0)-\bar u_n(0,0)\notag \\ &=-\int^0_{\bar \eta} \partial_\eta \bar u_n (\bar \xi,\eta) d\eta-\int^0_{\bar \xi} \partial_\xi \bar u_n(\xi, 0)d\xi.
\end{eqnarray}

From (\ref{xieta}), $\partial_\eta \bar u_n$ is monotonic in $\xi$. So, for example for $\eta=\bar\eta$, as $\bar \xi\le0\le \xi_0$, we have :
$$ \partial_{\eta } \bar u_n(\bar \xi,\bar\eta)\le \partial_{\eta  } \bar u_n(0,\bar\eta)\le \partial_{\eta  } \bar u_n( \xi_0,\bar\eta).$$
 
Similarly, for any $\eta \in (\bar \eta, 0)$, we can bound from above the function $\partial_\eta \bar u_n(\bar\xi,\eta)$ by its value at
the point $(\xi^*(\eta),\eta)$, which is the projection  of
$(\bar\xi,\eta)$ on the axis $\{s=0\}$ in parallel to the axis $\xi$ (as $\bar\xi\le0\le\xi^*(\eta)$).

By the same way, from (\ref{xieta}), $\partial_\xi \bar u_n$ is monotonic in $\eta$.
As $\bar\eta\le0\le \eta_0$, we can bound, for $\xi\in(\bar \xi,0)$, $\partial_\xi \bar u_n(\xi,0)$ by its value at the point $(\xi,\eta^*(\xi))$, which is the projection of $(\xi,0)$
on the axis $\{s=0\}$ in parallel to the axis $\eta$ ($0<\eta^*(\xi)$).
So it follows that:

\begin{equation}\label{equadi2}
\left.
\begin{array}{l}
 \partial_{\eta } \bar u_n(\bar \xi,\eta)\le \partial_{\eta  } \bar u_n( \xi^*(\eta),\eta),\;\forall \eta \in (\bar \eta, 0)\\
 \partial_{\xi} \bar u_n(\xi,0)\le \partial_{\xi} \bar u_n(\xi,\eta^*(\xi)),\;\forall \xi \in (\bar \xi, 0)
\end{array}
\right .
\end{equation}

By a straightforward geometrical construction, we see that the coordinates of $(\xi^*(\eta),\eta)$ and $(\xi,\eta^*(\xi))$, 
in the original set of variables $\{y,s\}$, are respectively $(x+t-\eta\sqrt{2},0)$ and $(x-t+\eta\sqrt{2},0)$. Both points are in $[-R,R]$.

Furthermore, we have from (\ref{21,5}):
\begin{align}
\left.
\begin{array}{l}\label{equadi3}
\partial_{\eta  } \bar u_n( \xi^*(\eta),\eta)=\frac{1}{2}(-\partial_t u_n-\partial_x u_n) (x+t-\eta\sqrt{2},0)=\frac{1}{2}(-u_{1,n}-\partial_x u_{0,n})(x+t-\eta\sqrt{2}),\\
\partial_{\xi} \bar u_n(\xi,\eta^*(\xi))=\frac{1}{2}(-\partial_t u_n+\partial_x u_n) (x-t+\eta\sqrt{2},0)=\frac{1}{2}(-u_{1,n}+\partial_x u_{0,n})(x-t+\eta\sqrt{2}).
\end{array}
\right .
\end{align}

Using (\ref{equadi3}), the Cauchy-Schwarz inequality and the fact that $u_{1,n}$ and $\partial_x u_{0,n}$ are uniformly bounded  in $  L^2(-R,R)$ since they are convergent, we have:
\begin{equation}\label{equadi4}
\left.
\begin{array}{l}
 \int_{\bar \eta}^0  \partial_{\eta  } \bar u_n( \xi^*(\eta),\eta)d\eta\le C(R),\\
 \int_{\bar\xi}^0  \partial_{\xi} \bar u_n(\xi,\eta^*(\xi))d\xi\le C(R).
\end{array}
\right .
\end{equation}

Using (\ref{equadi1}), (\ref{equadi2}) and (\ref{equadi4}), we reach to conclusion of Lemma \ref{lab}.

\end{proof}

Let us show the following:
\begin{lem}{\bf (Convergence of $u_n$ as $n\rightarrow \infty$)}\label{xt}
 Consider $(x,t)\in \mathbb{R}\times [0,+\infty)$. We have the following:
\begin{itemize}
 \item  [$-$] If $t>T(x)$, then $u_n(x,t)\rightarrow +\infty$,
 \item  [$-$] If $t<T(x)$, then $u_n(x,t)\rightarrow u(x,t)$.
\end{itemize}
\end{lem}
\begin{proof} We claim that it is enough to show the convergence for a subsequence. Indeed, this is clear from the fact that the limit is explicit and doesn't depend on the subsequence. Consider $(x,t)\in \mathbb{R}\times [0,+\infty),$ up to extracting a subsequence, there is $l(x,t)\in  \mathbb{\overline R}$ such that $u_n(x,t)\rightarrow l(x,t)$ as $n\rightarrow \infty$.\\
Let us show that $l\neq -\infty$. Since $F_n(u)\ge 0$, it follows that $u_n(x,t)\ge  \underline u_n(x,t)$ where 
\begin{equation}\label{28.5} \left\{
\begin{array}{l}
\displaystyle\partial^2_{t} \underline u_n = \partial^2_{x}\underline  u_n, \\
\underline u_n(0)=u_{0,n} \mbox{ and } \partial_t \underline u_n(0) = u_{1,n}.
\end{array}
\right . 
\end{equation}
Since $\underline u_n\in L^\infty_{loc}(\mathbb{R}^+,H^1(-R,R))\subset L^\infty_{loc}(\mathbb{R}^+,L^\infty(-R,R))$, for any $R>0$, 
from the fact that $(u_{0,n},u_{1,n})$ is convergent in $H^1_{loc}\times L^2_{loc}$, it follows that $l(x,t)\ge \limsup_{n\rightarrow +\infty}\underline u_n(x,t)>-\infty$.\\

Note from the fact that $F_n(u)\le e^u$ that we have
\begin{eqnarray}\label{15}
 \forall x\in \mathbb{R}, t< T(x),\,u_n(x,t)\le u(x,t).
\end{eqnarray}
Introducing $R=|x|+t+1$, we see by definition (\ref{6.5}) of $D_R$ that $(x,t)\in D_R$.
Let us handle two cases in the following:

\medskip
\noindent{\bf Case 1:} $t< T(x)$\\
Let us introduce $v_n$ the solution of 
\begin{equation*} \left\{
\begin{array}{l}
\displaystyle\partial^2_{t} v_n = \partial^2_{x}v_n+e^{v_n}, \\
v_n(0)=u_{0,n} \mbox{ and } \partial_t v_n(0) = u_{1,n} \in H^1_{loc,u}\times L^2_{loc,u}.
\end{array}
\right . 
\end{equation*}
From the local Cauchy theory in $ H^1_{loc,u}\times L^2_{loc,u}$ and the Sobolev embedding, we know that
\begin{eqnarray}\label{triangle1}
 v_n\rightarrow u \mbox{ uniformly as } n\rightarrow \infty \mbox{ in compact sets of }D.
\end{eqnarray}

Let us consider
$$\tilde K=\overline{ K_-(x,(t+T(x))/2})$$
and $\tilde M=\max_{(y,s)\in \tilde K} |u(y,s)|<+\infty$, since $\tilde K$ is a compact set in $D$.\\
From (\ref{triangle1}), we may assume $n$ large enough, so that $||u_{0,n}-u_0||_{L^{\infty}(\tilde K\cap \{t=0\})}\le 1$,\\

\begin{equation}\label{5d}
 \sup_{(y,s)\in \tilde K} |v_n(y,s)|\le \tilde M+1
\end{equation}
and
\begin{eqnarray}\label{4d}
 n\ge \tilde M+3.
\end{eqnarray}
In particular,
\begin{eqnarray}\label{dd}
 ||u_{0,n}||_{L^{\infty}(\tilde K\cap \{t=0\})}\le \tilde M+1.
\end{eqnarray}
We claim that 
\begin{equation}\label{7d}
 \forall (y,s)\in \tilde K, |u_n(y,s)|\le \tilde M+2.
\end{equation}
Indeed, arguing by contradiction, we may assume from (\ref{dd}) and continuity of $u_n$ that
\begin{eqnarray}\label{3d}
 \forall s\in [0,\tilde t_n],\, ||u_n(s)||_{L^{\infty}(\tilde K\cap \{t=s\})}\le \tilde M+2,
\end{eqnarray}
and
\begin{eqnarray}\label{6d}
||u_n(\tilde t_n)||_{L^{\infty}(\tilde K\cap \{t=\tilde t_n\})}= \tilde M+2,
\end{eqnarray}
for some $\tilde t_n\in (0,\frac{t+T(x)}{2})$.\\
From (\ref{4d}), (\ref{3d}) and the definition (\ref{16.3}) of $F_n$, we see that
\begin{eqnarray*}
 \forall (y,s) \in \tilde K \mbox{ with }s\le \tilde t_n, F_n(u_n(y,s))=e^{u_n(y,s)}.
\end{eqnarray*}
Therefore, $u_n$ and $v_n$ satisfy the same equation with the same initial data on $\tilde K\cap \{s\le \tilde t_n\}$. From uniqueness of the solution to the Cauchy problem, we see that
\begin{eqnarray*}
 \forall (y,s) \in \tilde K \mbox{ with }s\le \tilde t_n, u_n(y,s)=v_n(y,s).
\end{eqnarray*}
A contradiction then follows from (\ref{5d}) and (\ref{6d}). Thus, (\ref{7d}) holds.\\

Again from the choice of $n$ in (\ref{4d}), we see that
\begin{eqnarray*}
 \forall (y,s) \in \tilde K , F_n(u_n(y,s))=e^{u_n(y,s)},
\end{eqnarray*}
hence, from uniqueness,
\begin{eqnarray*}
 \forall (y,s) \in \tilde K , u_n(y,s)=v_n(y,s).
\end{eqnarray*}

From (\ref{triangle1}), and since $(x,t)\in \tilde K$, it follows that $u_n(x,t)\rightarrow u(x,t)$ as $n\rightarrow \infty$.

\medskip
\noindent{\bf Case 2:} $t> T(x)$\\
Assume by contradiction that $l<+\infty.$ From Lemma \ref{lab}, it follows that
\begin{eqnarray*}
 \forall (y,s)\in \overline{K^-(x,t)},\, u_n(y,s)\le u_n(x,t)+C(R).
\end{eqnarray*}
For $n\ge n_0$ large enough, this gives $u_n(y,s)\le l+1+ C(R).$\\
If $M=E(l+1+C(R))+1,$ then 
$$\forall n\ge \max (M, n_0),\,\forall (y,s)\in \overline{K^-(x,t)}, F_n(u_n(y,s))=e^{u_n(y,s)},$$ and $u_n$ satisfies (\ref{waveq}) in $\overline{K^-(x,t)}$ with initial data $(u_{0,n},u_{1,n})\rightarrow (u_0,u_1)\in H^1\times L^2(K^-(x,t)\cap\{t=0\})$.
From the finite speed of propagation and the continuity of solutions to the Cauchy problem with respect of initial data, it follows that $u_n$ and $u$ are both defined in $K^-(x,t)$ for $n$ large enough, in particular $u$ is defined at $(x,s)$ with $T(x)<s<t$ with $u=u_n$ in $\overline{K^-(x,t)}$. Contradiction with the expression of the domain of definition (\ref{domaine-de-definition}) of $u$.
\end{proof}

\section{Energy estimates}\label{sec2}
In this section, we use some localized energy techniques from Shatah and Struwe \cite{MR1674843} to derive a non-blow-up criterion which will give the lower bound in Theorem \ref{th}. More precisely, we give the following:
\begin{prop}\emph{\bf(Non blow-up criterion for a semilinear wave equation)}\label{theor}

$\forall c_0>0$, there exist $M_0 (c_0)>0$ and $M (c_0)>0$ such that, if
\begin{equation} {\rm (H)}:\left\{
\begin{array}{l}
||\partial_x u_0||_{L^2(-1,1)}^{2}+||u_1||_{L^2(-1,1)}^{2}\leq c_0^2\\
\forall |x|<1,\; u_0(x)\leq M_0,
\end{array}
\right .\label{H}
\end{equation}
then equation (\ref{waveq}) with initial data $(u_0,u_1)$ has a unique solution $(u, \partial_t u)\in C([0, 1), H^1 \times L^2(|x|<1-t))$ such that for all $t \in [0, 1),$ we have:\\
\begin{equation}\label{p1}
 ||\partial_x u (t)||_{L^2(|x|<1-t)}^{2}+||\partial_t u (t)||_{L^2(|x|<1-t)}^{2}\leq 2 c_0^2 ,
\end{equation}
and\\
\begin{equation}\label{p2}
\forall |x|<1-t ,\; u(x, t)\leq M .
 \end{equation}
\end{prop}
\medskip

Note that here, we work in the space $ H^1_{loc} \times L^2_{loc}$ which is larger than the space $H^{1}_{loc,u}\times L^2_{loc,u}$ which is adopted elsewhere for equation (\ref{waveq}). Before giving the proof of this result, let us first give the following corollary, which is a direct consequence of Proposition \ref{theor}.
\begin{cor}\label{pro}
There exists $\bar\epsilon_0>0$ such that if
\begin{equation}\label{soleil}
\int_{-1}^{1} (u_1(x))^2+(\partial_x u_0(x))^2+e^{u_0(x)} \, dx \leq \bar\epsilon_0,
\end{equation}
then the solution $u$ of equation (\ref{waveq}) with initial data $(u_0,u_1)$ doesn't blow up in the cone $\mathcal{C}_{0,1,1}$.
\end{cor}
Let us first derive Corollary \ref{pro} from Proposition \ref{theor}.
\begin{proof}[Proof of Corollary \ref{pro} assuming that Proposition \ref{theor} holds]:
 From (\ref{soleil}), if $\bar \epsilon_0\le 1$ we see that
 \begin{align} 
 \int_{-1}^1\left((u_1(x))^2+(\partial_x u_0(x))^2\right) dx &\le \bar \epsilon_0\le 1,\label{solei}\\ \notag
\int_{-1}^1 e^{u_0(x)}dx &\le \bar \epsilon_0.
\end{align}
Therefore, for some $x_0\in (-1,1)$, we have $2e^{u_0(x_0)}=\int_{-1}^1 e^{u_0(x)} dx \le \bar \epsilon_0$, hence $u_0(x_0)\le \log\frac{\bar \epsilon_0}{2}$.
Using (\ref{solei}), we see that for all $x\in (-1,1),$
\begin{align*}
 u_0(x)=u_0(x_0)+\int_{x_0}^x \partial_x u_0\le u_0(x_0)+\sqrt{2}\left( \int_{-1}^1 (\partial_x u_0(x))^2dx\right)^\frac{1}{2}\le \log\frac{\bar \epsilon_0}{2}+\sqrt{2\bar \epsilon_0}\le M_0(1),
\end{align*}
defined in Proposition \ref{theor}, provided that $\bar \epsilon_0$ is small enough. Therefore, the hypothesis (H) of Proposition \ref{theor} holds with $c_0=1$, and so does its conclusion. This concludes the proof of Corollary \ref{pro}, assuming that Proposition \ref{theor} holds.
\end{proof}
Now, we give the proof of Proposition \ref{theor}.
\begin{proof}[Proof of Proposition \ref{theor}]
Consider $c_0>0$ and introduce
\begin{align*}
 M_0=\log(\frac{c_0^2}{16})-c_0\sqrt{2}-\frac{c_0^2}{8}\mbox{ and }M(c_0)=\log(\frac{c_0^2}{16}).
\end{align*}
Then, we consider $(u_0,u_1)$ satisfying hypothesis (H). From the solution of the Cauchy problem in $H^1_{loc}\times L^2_{loc}$, which follows exactly by the same argument as in the space $H^{1}_{loc,u}\times L^2_{loc,u}$ presented in Section \ref{cauchy-Problem}, there exists $t^*\in (0,1]$ such that equation (\ref{waveq}) has a unique solution with $(u,\partial_t u)\in C([0,t^*),H^1\times L^2(|x|<1-t))$. Our aim is to show that $t^*=1$ and that (\ref{p1}) and (\ref{p2}) hold for all $t\in [0,1)$.

Clearly, from the solution of the Cauchy problem, it is enough to show that (\ref{p1}) and (\ref{p2}) hold for all $t\in [0,t^*)$, so we only do that in the following:

Arguing by contradiction, we assume that there exists at least some time $t \in [0, t^*)$ such that either (\ref{p1}) or (\ref{p2}) doesn't hold. If $\bar t$ is the lowest possible $t$, then, we have from continuity,
either 
$$||\partial_x u (\bar t)||_{L^2(|x|<1-\bar t)}^{2}+||\partial_t u (\bar t)||_{L^2(|x|<1-t)}^{2} =2 c_0,$$
or
$$\exists |x_0|< 1-\bar t,\mbox{such that } u(x_0,\bar t)= M. $$
Note that since (\ref{p1}) holds for all $t\in [0,\bar t)$, it follows that
\begin{eqnarray}\label{to}
 \forall t \in [0,\bar t),\forall |x|<1-t, u(x,t)\le M=\log(\frac{c_0^2}{16}).
\end{eqnarray}

Following the alternative on $\bar t$, two cases arise in the following:

\medskip
\noindent \emph{\bf{Case 1}}: $||\partial_x u (\bar t)||_{L^2(|x|<1-\bar t)}^{2}+||\partial_t u (\bar t)||_{L^2(|x|<1-\bar t)}^{2} = 2c_0^2 $.
\\
Referring to Shatah and Struwe \cite{MR1674843}, we see that:
\begin{eqnarray}\label{equashatah}
\int_{|x|<1-\bar t}(\frac{1}{2}(\partial_x u^2+\partial_t u^2)-e^u )\; dx&-&
\int_{|x|<1}(\frac{1}{2}(\partial_x u_0^2+ u_1^2)-e^{u_0}) \; dx \notag\\ &=&\int_{\Gamma} (e^u-\frac{1}{2} |\partial_x u-\frac{x}{|x|}\partial_t u |^2)\; d\sigma,
\end{eqnarray}
where $$\Gamma = \{ (x, t) \in \mathbb{R}\times \mathbb{R}_+,\; \mbox{such that } |x|=1-t\} \cap [0,\bar t] $$
Using (\ref{to}), it follows that
$$\int_{|x|<1-\bar t}e^{u(x,\bar t)}dx\le \int_{|x|<1-\bar t}e^M\le \frac{c_0^2}{8}.$$
$$\int_{\Gamma}e^ud\sigma\le \int_0^{\bar t}(e^{u(1-t,t)}+e^{u(t-1,t)})dt\le\frac{c_0^2}{8}.$$
Therefore, from (\ref{equashatah}) and (\ref{H}), we write
\begin{align*}
 \int_{|x|<1-\bar t}(\partial_x u)^2+(\partial_t u)^2 dx & \le  \int_{|x|<1}(\partial_x u_0)^2+(u_1)^2+ \int_{|x|<1-\bar t}e^{u(x,\bar t)}dx+\int_{\Gamma}e^u d\sigma\\
&\le c_0^2+\frac{3}{8}c_0^2<2c_0^2,
\end{align*}

which is a contradiction.\\
\medskip

\noindent \emph{\bf{Case 2}}:\;$\exists x_0\in (- (1-\bar t),1-\bar t),\, u(x_0,\bar t)= M. $\\
\noindent Recall Duhamel's formula:
\begin{eqnarray}\label{duhamel1}
 \forall |x|<1-\bar t,\notag\\
u(x, \bar t)&=&\frac{1}{2}(  u_0(x-\bar t)+ u_0(x+\bar t)+ \frac{1}{2}\int_{x-\bar t}^{x+\bar t} u_1 \notag\\&+&\frac{1}{2}\int_0^{\bar t}\int_{x-\bar t+\tau}^{x+\bar t-\tau} e^{u(z,\tau)}\;dz d\tau.
\end{eqnarray}
From (H), we write,
$$\int_{x-\bar t}^{x+\bar t} u_1 \; dx \leq\left(\int_{-1}^1 u_1^2\right)^\frac{1}{2} 2 \sqrt{2} \le c_0\sqrt{2} . $$
From (\ref{to}), we write
\begin{eqnarray*}
 \int_0^{\bar t}\int_{z-\bar t+\tau}^{z+\bar t -\tau}e^{u(z,\tau)}dz d\tau\le \int_0^{\bar t}\int_{z-\bar t+\tau}^{z+\bar t -\tau}\frac{c_0^2}{16}\le \frac{c_0^2}{8}.
\end{eqnarray*}
Since $u_0(x\pm \bar t)\le M_0=\log(\frac{c_0^2}{16})-c_0\sqrt{2}-\frac{c_0^2}{8}$, it follows from (\ref{duhamel1}) that
$$u(x,\bar t)\le M_0+ c_0 \frac{\sqrt{2}}{2}+\frac{c_0^2}{16}<\log(\frac{c_0^2}{16})=M,$$
and a contradiction follows.

This concludes the proof of Proposition \ref{theor}. Since we have already derived Corollary \ref{pro} from Proposition \ref{theor}, this is also the conclusion of the proof of Corollary \ref{pro}.

\end{proof}
\section {ODE type estimates }\label{sec3}
In this section, we extend the work of Godin in \cite{MR1854257}. In fact, we show that his estimates holds for more general initial data. As in the introduction, we consider $u(x,t)$ a non global solution of equation (\ref{waveq}) with initial data $(u_0,u_1)\in H^{1}_{loc,u}\times L^{2}_{loc,u}$.
This section is organized as follows:\\
-In the first subsection, we give some preliminary results and we show that the solution goes to $+\infty$ on the graph $ \Gamma$.\\
-In the second subsection, we give and prove upper and lower bounds on the blow-up rate.
\subsection{Preliminaries}

In this subsection, we first give some geometrical estimates on the blow-up curve (see Lemmas \ref{co}, \ref{so} and \ref{fo}). Then, we use equation (\ref{waveq}) to derive a kind of maximum principle in light cones (see Lemma \ref{3.2}), then, a lower bound on the blow-up rate (see Proposition \ref{ll}).
\medskip

We first we give the following geometrical property concerning the distance to \\$\{t=T(x)\},$ the boundary of the domain of definition of $u(x,t)$.
\begin{lem} {\bf (Estimate for the distance to the blow-up boundary).}\label{co} For all $(x,t)\in D$, we have
\begin{eqnarray}
   \frac{1}{\sqrt2}(T(x)-t)\le d((x,t),\Gamma) \le T(x)-t,
\end{eqnarray}
where $d((x,t),\Gamma)$ is the distance from $(x,t)$ to $\Gamma$. 
\end{lem}
\begin{proof}
Note first by definition that
$$d((x,t),\Gamma)\le d((x,t),(x,T(x))=T(x)-t.$$
Then, from the finite speed of propagation, $\Gamma$ is above $\mathcal{C}_{x,T(x),1}$, the backward light cone with vertex $(x,T(x))$. Since $(x,t)\in\mathcal{C}_{x,T(x),1}$, it follows that
$$d((x,t),\Gamma)\ge d((x,t),\mathcal{C}_{x,T(x),1})=\frac{\sqrt{2}}{2}(T(x)-t).$$
This concludes the proof of Lemma \ref{co}.
\end{proof}
Then, we give a geometrical property concerning distances, specific for non-characteristic points.
\begin{lem} {\bf (A geometrical property for non-characteristic points)}\label{so}
 Let $a \in \mathcal{R}$. There exists $c:=C(\delta)$, where $\delta=\delta(a)$ is given by (\ref{4}), such that for all $(x,t)\in\mathcal{C}_{a,T(a),1}$, 
\begin{align*}
 \frac{1}{c}\le \frac{T(x)-t}{T(a)-t}\le c.
\end{align*}
\end{lem}
\noindent{\bf Remark}: From Lemma \ref{co}, it follows that
$$\frac{1}{\bar c}\le \frac{d((x,t),\Gamma)}{d((a,t),\Gamma)}\le \bar c$$
whenever $a\in \mathcal{R}$ and $(x,t)\in\mathcal{C}_{a,T(a),1}$.
\begin{proof}
Let $a$ be a non-characteristic point. We recall form condition (\ref{4}) that
\begin{equation*}
 \exists \,\delta = \delta (a) \in (0,1) \mbox{ such that } u \mbox{ is defined on }\mathcal{C}_{a,T(a), \delta}.
\end{equation*}
 Let $(x,t)$ be in the light cone with vertex $(a,T(a))$. Using the fact that the blow-up graph is above the cone $\mathcal{C}_{a,T(a), \delta}$ and the fact that $(x,t)\in  \mathcal{C}_{a,T(a), 1},$ we see that
\begin{align}\label{soo}
 T(x)-t\ge T(a)-\delta|x-a|-t\ge (T(a)-t)(1-\delta).
\end{align}
In addition, as $\Gamma$ is a 1-Lipchitz graph, we have
$$T(x)\le T(a)+|x-a|,$$
so, for all $(x,t)\in  \mathcal{C}_{a,T(a), 1}$
\begin{align}\label{sooo}
 T(x)-t\le T(a)+|x-a|-t\le 2(T(a)-t)
\end{align}
From (\ref{soo}) and (\ref{sooo}), there exists $c=c(\delta)$ such that
\begin{align*}
\frac{1}{c}\le \frac{T(x)-t}{T(a)-t}\le c.
\end{align*}
This concludes the proof of Lemma \ref{so}.
\end{proof}
Finally, we give the following coercivity estimate on the distance to the blow-up curve, still specific for non-characteristic points.
\begin{lem}\label{fo}
Let $x\in \mathcal{R}$ and $t\in [0,T(x))$. For all $\tau \in [0,t)$ and $j=1,2$ we have
\begin{eqnarray}\label{c}
d((z_j,w_j),\Gamma)\ge \frac{1}{C} (d(( x, t),\Gamma)+|( x, t)-(z_j,w_j)|),
\end{eqnarray}
where $(z_1,w_1)=( x+ t-\tau,\tau)$ and $(z_2,w_2)=( x- t+\tau,\tau)$.
\end{lem}
{\bf Remark: } Note that $(z_j,w_j)$ for $j=1,2$ lay on the backward light cone with vertex $(x,t)$.
\begin{proof}
Consider $x\in\mathcal{R}$, $t\in[0,T(x))$. By definition, there exists $\delta\in(0,1)$ such that $\mathcal{C}_{x,T(x),\delta}\subset D$. We will prove the estimate for $j=1$ and $\tau\in [0,t)$, since the the estimate for $j=2$ follows by symmetry. In order to do so, we introduce the following notations, as illustrated in figure 1:
$M=(x,T(x))$, $M_0=(x,t)$ and $M_1=(z_1,w_1)=( x+ t-\tau,\tau)$, which is on the left boundary of the backward light cone $\mathcal{C}_{x,t,1}$,
$N_1$ the orthogonal projection of $M_1$ on the left boundary of the cone $\mathcal{C}_{x,T(x),\delta}$, $P_1$ the orthogonal projection of $M_0$ on $[N_1,M_1]$. Note that the quadrangle $M_0N_0N_1P_1$ is a rectangle. If $\alpha$ is such that $\tan \alpha=\delta$ and $\beta=\widehat{PM_1M_0}$, then we see from elementary considerations on angles that $\beta=\alpha+\frac{\pi}{4}$ and $\widehat{N_0M_0M}=\alpha$.\\
Therefore, using Lemma \ref{co}, and the angles on the triangle $M_0N_0M$, we see that:
\begin{align}\label{49.5}
 d((x,t),\Gamma)\le T(x)-t=MM_0=\frac{M_0N_0}{\cos \alpha}=\frac{N_1 P_1}{\cos \alpha}.
\end{align}
Moreover, since the blow-up graph is above the cone $\mathcal{C}_{x,T(x),\delta}$, it follows that
$$d((z_1,w_1),\Gamma)\ge M_1N_1.$$
In particular,
\begin{align}\label{ppf}
 M_1N_1=N_1P_1+P_1M_1=N_1P_1+\cos(\frac{\pi}{4}+\alpha)M_1M_0.
\end{align}
Since $0<\delta<1$, hence $0<\alpha<\frac{\pi}{4}$, it follows that $\cos(\frac{\pi}{4}+\alpha)>0$. Since $M_1M_0=|(z_1,w_1)-(x,t)|$, the result follows from (\ref{49.5}) and (\ref{ppf}).
 \begin{figure}[h!]
\centering
\resizebox{10cm}{!}{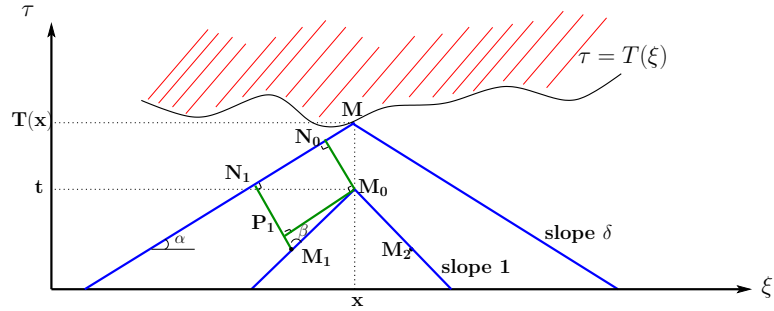}
\caption{Illustration for the proof of (\ref{c})}\label{cone}
\end{figure}

By the same way, we can prove this for the other point $M_2=(z_2,w_2)$, which gives (\ref{c}).\\
\end{proof}
Now, we give the following corollary from the approximation procedure in Lemmas \ref{lab} and \ref{xt}.
\begin{lem}\label{3.2}{\bf (Uniform bounds on variations of $u$ in cones).} For any $R>0$, There exists a constant $C(R)>0$ such that if $(x,t)\in D\cap \overline{D_R}$ then
 $$u(y,s)\ge u(x,t)-C(R),\; \forall (y,s)\in D\cap \overline{K^+_R(x,t)},$$
 $$u(y,s)\le u(x,t)+C(R),\; \forall (y,s)\in \overline{K^-(x,t)}.$$
where the cones $K^\pm$ and $K^\pm_R$ are defined in (\ref{16.55}).
\end{lem}
{\bf Remark}: The constant $C(R)$ depends also on $u_0$ and $u_1$, but we omit this dependence in the sequel.\\
In the following, we give a lower bound on the blow-up rate and we show that $u(x,t)\rightarrow +\infty$ as $t\rightarrow T(x)$.
\begin{prop} {\bf (A general lower bound on the blow-up rate)}\label{ll}  

\item{(i)} If $(u_0,u_1)\in W^{1,\infty}\times L^{\infty}$, then for all $R>0$, there exists $C(R)>0$ such that for all $(x,t)\in D\cap D_R,$ 
$$ d((x,t),\Gamma)e^{u(x,t)} \ge C.$$
In particular, for all $(x,t)\in D\cap \overline{D_R}$, $ u(x,t)\rightarrow +\infty$ as $d((x,t),\Gamma)\rightarrow 0$.

\item{(ii)} If, we only have $(u_0,u_1)\in H^{1}_{loc,u}\times L^{2}_{loc,u}$, then for all $R>0$, there exists $C(R)>0$ such that for all $(x_0,t)\in D\cap D_R,$ 
$$\frac{1}{T(x_0)-t}\int_{|x-x_0|<T(x_0)-t} e^{-u(x,t)}dx \le C(R)\sqrt{d(x_0,t)}.$$
In particular, $e^{-u}$ converges to $0$ in average over slices of the light cone, as $d(x_0,t)\rightarrow 0$.

\end{prop}
{\bf Remark}: Near non characteristic points, we are able to derive the optimal lower bound on the blow-up rate. See item $(ii)$ of Proposition \ref{p}.
\begin{proof}[Proof of Proposition \ref{ll}]\item{(i)} Clearly, the last sentence in item $(i)$ follows from the first, hence, we only prove the first.

 Let $R>0$ and $(x,t)\in D\cap D_R.$ Using the approximation procedure defined in (\ref{16.5}), we write $u_n=\underline u_n+\tilde u_n$ with:
$$\underline u_n(x,t)=\frac{1}{2}\left(u_{0,n}(x-t)+u_{0,n}(x+t)\right)+\frac{1}{2}\int_{x-t}^{x+t}u_{1,n}(\xi)d\xi,$$
$$\tilde u_n(x,t)=\frac{1}{2}\int_0^t\int_{x-t+\tau}^{x+t-\tau}F_n(u_n(z,\tau))dz d\tau$$

(Note that $\underline u_n$ was already defined in (\ref{28.5})).

Since $F_n\ge0$ from (\ref{16.3}), it follows that 
\begin{eqnarray}\label{carreau1}
u_n(x,t)\ge \underline u_n(x,t)\ge -C(R)\mbox{ for all }(x,t)\in D\cap D_R.
\end{eqnarray}
Differentiating $\underline u_n$, we see that
\begin{align}\label{50.5}
  \partial_t \underline u_n(x,t)&=\frac{1}{2}\left( \partial_x u_{0,n}(x+t)-\partial_x u_{0,n}(x-t)\right)+\frac{1}{2}\left(  u_{1,n}(x+t)+ u_{1,n}(x-t)\right)\notag\\
&\le || \partial_x u_{0,n}||_{L^\infty(-R,R)}+|| u_{1,n}||_{L^\infty(-R,R)}\le C(R).
\end{align}
Differentiating $\tilde u_n$, we get
\begin{align*}
 \partial_t \tilde u_n(x,t) &= \frac{1}{2} \int_0^t \left( F_n(u_n(x-t+\tau,\tau)+F_n(u_n(x+t-\tau,\tau))\right) d \tau \\
  &\le \frac{1}{2} \int_0^t \left( e^{u_n(x-t+\tau,\tau)}+e^{u_n(x+t-\tau,\tau)}\right) d \tau 
\end{align*}
since $F_n(u)\le e^{u}$. Since $u_n(x-t+\tau,\tau)\le u_n(x,t)+C(R)$ and $u_n(x+t-\tau,\tau)\le u_n(x,t)+C(R)$ from Lemma \ref{lab}, it follows that

\begin{align}\label{50.6}
\partial_t \tilde u_n(x,t) \le   c t e^{ u_n(x,t)} \le C( R) e^{ u_n(x,t)}.
\end{align}
Therefore, using (\ref{carreau1}) we see that
 \begin{align*}
  \partial_t  u_n(x,t)= \partial_t \underline u_n(x,t)+ \partial_t \tilde u_n(x,t)
  \le C(R)+C(R)e^{u_n(x,t)}\le C(R) e^{u_n(x,t)},
 \end{align*}
hence,
\begin{align}\label{1}
  \partial_t  u_n(x,t) e^{-u_n(x,t)}\le C(R).
 \end{align}
Integrating (\ref{1}) on any interval $[t_1,t_2]$ with $0\le t_1<T(x)<t_2$, we get $e^{-u_n(x,t_1)}-e^{-u_n(x,t_2)}\le C (t_2-t_1)$. Making $n \rightarrow \infty $ and using Lemma $\ref{xt} $ we see that  $ e^{-u(x,t_1)}\le C (t_2-t_1)$.\\
Taking $t_1=t$ and making $t_2\rightarrow T(x)$, we get $ e^{-u(x,t)}\le C (T(x)-t)$. Using Lemma \ref{co} concludes the proof of item $ (i)$ of Proposition \ref{ll}.

\item{(ii)} if $(u_0,u_1)\in H^{1}_{loc,u}\times L^{2}_{loc,u}$, then small modification in the argument of item $(i)$ gives the result. Indeed, if $t_0\in [0,T(x_0))$,
$$a_0=x_0-(T(x_0)-t_0),\; b_0=x_0+(T(x_0)-t_0),$$
and $t\ge 0$, we write from (\ref{carreau1}) and (\ref{50.5})
 \begin{align*}
  \int_{a_0}^{b_0} \partial_t \underline u_n(x,t) e^{-u_n(x,t)}dx\le C(R) \sqrt{b_0-a_0}.
 \end{align*}
Furthermore, from (\ref{50.6}) we write
 \begin{align*}
  \int_{a_0}^{b_0} \partial_t \tilde u_n(x,t) e^{-u_n(x,t)}dx\le C(R) (b_0-a_0).
 \end{align*}
Therefore, it follows that

 \begin{align}\label{50.7}
 -\frac{d}{dt} \int_{a_0}^{b_0} e^{-u_n(x,t)}dx= \int_{a_0}^{b_0} \partial_t u_n(x,t) e^{-u_n(x,t)}dx\le C(R) \sqrt{b_0-a_0}.
 \end{align}
Integrating (\ref{50.7}) on interval $(t_0,t_0')$, where
 \begin{align}\label{51.5}
  t_0'=2T(x_0)-t_0,
 \end{align}
we get
 \begin{align}\label{50.8}
  \int_{a_0}^{b_0} e^{-u_n(x,t_0)}dx-\int_{a_0}^{b_0}  e^{-u_n(x,t_0')}dx\le C(R)\sqrt{b_0-a_0}(t_0'-t_0)=2\sqrt{2}C(R)(T(x_0)-t_0)^\frac{3}{2}.
 \end{align}

Since $x\mapsto T(x)$ is 1-Lipschitz and $T(x_0)$ is the middle of $[t_0,t_0']$, we clearly see that the segment $[a_0,b_0]\times \{t_0'\}$ lays outside the domain of definition of $u(x,t)$, using Lemma \ref{xt}, we see that
$$\int_{a_0}^{b_0}  e^{-u_n(x,t_0')}dx\rightarrow 0 \mbox{ as }n\rightarrow +\infty,$$
on the one hand (we use Lebesgue Lemma together with the bound (\ref{carreau1})).

On the other hand, similarly, we see that

$$\int_{a_0}^{b_0}  e^{-u_n(x,t_0)}dx\rightarrow \int_{a_0}^{b_0}  e^{-u(x,t_0)}dx \mbox{ as }n\rightarrow +\infty.$$

Thus, the conclusion follows from (\ref{50.8}), together with Lemma \ref{co}, this concludes the proof of Proposition \ref{ll}.

\end{proof}
\subsection{ The blow-up rate}
This subsection is devoted to bound the solution $u$. We have obtained the following result.
\begin{prop} \label{p} For any $R>0$, there exists $C(R)>0$, such that:
  \item{(i)} {\bf (Upper bound on $u$)} For all $(x,t)\in D\cap D_R$ we have
 \begin{eqnarray*}
 e^ud((x,t),\Gamma)^2\le C.
 \end{eqnarray*}
 \item{(ii)} {\bf (Lower bound on $u$)} If in addition $(u_0,u_1)\in W^{1,\infty}\times L^{\infty}$ and $x$ is a non-characteristic point, then for all $(x,t)\in D\cap \overline{ D_R}$
 \begin{eqnarray*}
 e^ud((x,t),\Gamma)^2\ge\frac{1}{C} .
 \end{eqnarray*}

\end{prop}
{\bf Remark}: In \cite{MR1854257}, Godin didn't use the notion of characteristic point, but the regularity of initial data was fundamental to have the result. In this work, our initial data are less regular, so we focused on the case of non-characteristic point in order to get his result.

\begin{proof}[Proof of Proposition \ref{p}]$ $\\
\noindent $(i)$ Consider $R>0$. We will show the existence of some $C(R)>0$ such that for any $(x,t_1)\in D\cap D_R,$ we have
$$e^u d((x,t),\Gamma)^2\le C(R).$$
Consider then $(x,t_1)\in D\cap D_R.$ Since $x\mapsto T(x)$ is 1-Lipschitz, we clearly see that
\begin{eqnarray}\label{*}
 (x,T(x))\in D_{\bar R}\mbox{ with }\bar R=2R+T(a)+1.
\end{eqnarray}
Consider now $t_2\in (t_1,T(x))$ to be fixed later. We introduce the square domain with vertices $(x,t_1),(x+\frac{t_2-t_1}{2},\frac{t_1+t_2}{2}),(x,t_2),(x-\frac{t_2-t_1}{2},\frac{t_1+t_2}{2})$. Let

$$T_{sup}=\{(\xi,t)\,|\,\frac{t_1-t_2}{2} <t<t_2,\; |x-\xi|<t_2-t\},$$
$$T_{inf}=\{(\xi,t)\,|\, t_1<t<\frac{t_1-t_2}{2},\; |x-\xi|<t-t_1\},$$
respectively the upper and lower half of the considered square. From Duhamel's formula, we write:
\begin{eqnarray*}
u(x, t_2)&=&\frac{1}{2}  u(x+\frac{t_2-t_1}{2},\frac{t_2+t_1}{2} )+\frac{1}{2}  u(x-\frac{t_2-t_1}{2},\frac{t_2+t_1}{2} )\notag\\&+& \frac{1}{2}\int_{x-\frac{t_2-t_1}{2}}^{x+\frac{t_2-t_1}{2}}\partial_t u (\xi,\frac{t_2-t_1}{2})\,d\xi+\frac{1}{2}\int_{T_{sup}} e^{u(\xi,t)}\;d\xi dt,
\end{eqnarray*}
and,
\begin{eqnarray*}
u(x, t_1)&=&\frac{1}{2}  u(x+\frac{t_2-t_1}{2},\frac{t_2+t_1}{2} )+\frac{1}{2}  u(x-\frac{t_2-t_1}{2},\frac{t_2+t_1}{2} )\notag\\&-& \frac{1}{2}\int_{x-\frac{t_2-t_1}{2}}^{x+\frac{t_2-t_1}{2}}\partial_t u (\xi,\frac{t_2-t_1}{2})\,d\xi+\frac{1}{2}\int_{ T_{inf}}e^{u(\xi,t)}\;d\xi dt.
\end{eqnarray*}
So,
\begin{eqnarray}\label{t}
&u(x, t_1)+ u(x, t_2)\\
&=u(x+\frac{t_2-t_1}{2},\frac{t_2+t_1}{2} )+ u(x-\frac{t_2-t_1}{2},\frac{t_2+t_1}{2} )+\frac{1}{2}\int_{ T_{inf}\cup T_{sup}}e^{u(\xi,t)}\;d\xi dt.\notag
\end{eqnarray}
Since the square $ T_{inf}\cup T_{sup}\subset D_{\bar R}$ from (\ref{*}), applying Lemma \ref{3.2}, we have for all $(\tilde x, \tilde t)\in T_{inf}\cup T_{sup}$ and for some $C(R)>0$: 
$$u(\tilde x, \tilde t)\ge u(x,t_1)-C.$$

Applying this to (\ref{t}), we get
$$u(x,t_2)\ge u(x,t_1)-2C+\frac{(t_2-t_1)^2}{4}e^{(u(x,t_1)-C)}.$$
Now, choosing $t_2=t_1+\sigma e^{-u(x,t_1)/2}$ where $\sigma= 2e^\frac{c}{2}\sqrt{\eta+2C}$, we see that\\
-either $(x,t)\notin D$\\
-or $(x,t_2)\in D$ and $u(x,t_2)\ge u(x,t_1)+1$\\
by the above-given analysis.
 
In the second case, we may proceed similarly and define for $n\ge3$ a sequence
\begin{align}\label{triangle}
t_n=t_{n-1}+\sigma e^{-u(x,t_{n-1})/2},
\end{align}
as long as $(x,t_{n-1})\in D$. Clearly, the sequence $(t_n)$ is increasing whenever it exists. Repeating between $t_n$ and $t_{n-1}$, for $n\ge3$, the argument we first wrote for $t_1$ and $t_2$, we see that
\begin{align}\label{carreau}
  u(x,t_n)\ge u(x,t_{n-1})+1,
\end{align}
as long as $(x,t_n)\in D$. Two cases arise then:\\
\medskip

-{\bf Case 1}: The sequence $(t_n)$ can be defined for all $n\ge1$, which means that
\begin{align}\label{trian}
 (x,t_n)\in D,\; \forall n\in \mathbb{N}^*.
\end{align}
In particular, (\ref{triangle}) and (\ref{carreau}) hold for all $n\ge2$.

If $t_\infty=\lim_{n\rightarrow \infty}t_n$, then, from (\ref{trian}), we see that $t_\infty\le T(x)$.

Since $u(x,t_n)\rightarrow +\infty$ as $n\rightarrow \infty$ from (\ref{carreau}), we need to have
$$t_\infty=T(x),$$
from the Cauchy theory. Therefore, using Lemma \ref{co}, (\ref{triangle}) and (\ref{carreau}), we see that
\begin{align*}
 d((x,t_1),\Gamma)\le T(x)-t_1=\sum\limits_{n=1}^\infty (t_{n+1}-t_n)\le \sigma \sum\limits_{n=1}^\infty  e^{-(u(x,t_1)+(n-1) )/2}\equiv C(R) e^{-u(x,t_1)/2},
\end{align*}
which is the desired estimate.\\
\medskip

-{\bf Case 2}: The sequence $(t_n)$ exists only for all $n\in[1,k]$ for some $k\ge2$. This means that $(x,t_k)\notin D$, that is $t_k\ge T(x)$.\\
Moreover, (\ref{triangle}) holds for all $n\in[2,k]$, and (\ref{carreau}) holds for all $n\in[2,k-1]$ (in particular, it is never true if $k=2$). As for case 1, we use Lemma \ref{co}, (\ref{triangle}) and (\ref{carreau}) to write
\begin{align*}
 d((x,t_1),\Gamma)&\le T(x)-t_1\le t_k-t_1=\sum\limits_{n=1}^{k-1}(t_{n+1}-t_n)\le \sigma \sum\limits_{n=1}^{k-1} e^{-(u(x,t_1)+(n-1) )/2}\\
&\le \sigma e^{-(u(x,t_1))/2}\sum\limits_{n=1}^{k-1} e^{(n-1) )/2}\equiv C(R) e^{-u(x,t_1)/2},
\end{align*}
which is the desired estimate. This concludes the proof of item $i)$ of Proposition \ref{p}.

\medskip

\noindent $(ii)$ Consider $R>0$ and $x$ a non-characteristic point such that $(x,t) \in D_R\cap D$. We dissociate $u$ into two parts $u=\bar u+\tilde u$ with:
$$\bar u(x,t)=\frac{1}{2}\left(u_0(x-t)+u_0(x+t)\right)+\frac{1}{2}\int_{x-t}^{x+t}u_1(\xi)d\xi,$$
$$\tilde u(x,t)=\frac{1}{2}\int_0^t\int_{x-t+\tau}^{x+t-\tau}e^{u(z,\tau)}dz d\tau.$$
Differentiating $\bar u$, we see that
\begin{align*}
  \partial_t \bar u(x,t)&=\frac{1}{2}\left( \partial_x u_0(x+t)-\partial_x u_0(x-t)\right)+\frac{1}{2}\left(  u_1(x+t)+ u_1(x-t)\right)\\
&\le || \partial_x u_0||_{L^\infty(-R,R)}+|| u_1||_{L^\infty(-R,R)}\le C(R),
\end{align*}
since $(x,t)\in D_R$. Consider now an arbitrary $a\in(\frac{1}{2},1)$. Since $u(x,t)\rightarrow +\infty$ as $d((x,t),\Gamma)\rightarrow 0$ (see Proposition \ref{ll} above), it follows that

   \begin{eqnarray}\label{cp} 
 e^{(a-1)  u(x,t)} \partial_t  \bar u(x,t)\le C(R) d((x,t),\Gamma)^{-2a+1}.
   \end{eqnarray}
Now, we will prove a similar inequality for $\tilde u$.
 Differentiating $\tilde u$, we see that
 \begin{eqnarray}\label{a}
  \partial_t \tilde u=\frac{1}{2}\int_0^t (e^{u(x-t+\tau,\tau)}+e^{u(x+t-\tau,\tau)}) d\tau.
 \end{eqnarray}
Using the upper bound in Proposition \ref{p}, part $(i)$, which is already proved and Lemma \ref{3.2}, we see that for all $(y,s)\in\overline{K^-(x,t)}, $
 \begin{eqnarray*}
  u(y,s)=(1-a)u(y,s)+a u(y,s)\le (1-a) (u(x,t)+C)+a (\log C-2 \log d((y,s),\Gamma)).
 \end{eqnarray*}
 So,
\begin{eqnarray}\label{b}
  e^{u(y,s)}\le C e^{(1-a)u(x,t)}(d((y,s),\Gamma))^{-2a}.
 \end{eqnarray}
Since $ x$ is a non-characteristic point, there exists $\delta_0\in (0,1)$ such that the cone $\mathcal{C}_{\bar x,T(\bar x),\delta_0}$ is below the blow-up graph $\Gamma$.

\noindent Applying $(\ref{b})$ and Lemma \ref{fo} to $(\ref{a})$, and using the fact that $|(x,t)-(z_1,w_1)|^2=2(\tau-t)^2$, we write (recall that $\frac{1}{2}<a<1$):
\begin{align*}
 \partial_t \tilde u &= \frac{1}{2} \int_0^t e^{u(z_1, w_1)}+e^{u(z_2, w_2)} d \tau \\
&\le  \frac{1}{2}\int_0^t Ce^{(1-a) u(x,t)}(d((z_1,w_1),\Gamma)^{-2a}+d((z_2,w_2),\Gamma)^{-2a})d \tau \\
&\le  C  e^{(1-a) u(x,t)}\int_0^t \left(  d((x,t),\Gamma)+|(x,t)-(z_1,w_1)|\right)^{-2a} +\left(  d((x,t),\Gamma)+|(x,t)-(z_1,w_1)| \right)^{-2a} d \tau \\
&\le  C  e^{(1-a) u(x,t)}\int_0^t \left( d((x,t),\Gamma)+\sqrt{2}(t-\tau)\right)^{-2a} d \tau \le C e^{(1-a) u(x,t)}\frac{1}{\sqrt 2 (2a-1)} d((x,t),\Gamma)^{-2a+1},
\end{align*}
which yields 
\begin{eqnarray}\label{ccp}
  e^{(a-1) u(x,t)} \partial_t \tilde u(x,t)\le C d((x,t),\Gamma)^{-2a+1}.
\end{eqnarray}
In conclusion, we have from (\ref{cp}), (\ref{ccp}) and Lemma \ref{co}
\begin{eqnarray}\label{A9}
  e^{(a-1) u(x,t)} \partial_t u(x,t)\le C d((x,t),\Gamma)^{-2a+1}\le C (T(x)-t)^{-2a+1}.
\end{eqnarray}
Since $u(x,t)\rightarrow+\infty$ as $d((x,t),\Gamma)\rightarrow 0$ from Proposition \ref{ll}, integrating (\ref{A9}) between $t$ and $T(x)$, we see that
$$e^{(a-1) u(x,t)}  \le C (T(x)-t)^{2-2a}.$$ Using again Lemma \ref{co}, we complete the proof of part $(ii)$ of Proposition \ref{p}.
\end{proof}
\section{Blow-up estimates for equation (\ref{waveq})}\label{sec4}
In this section, we prove the three results of our paper: Theorem \ref{th}, Theorem \ref{NEW} and Proposition \ref{2.1}. Each proof is given in a separate subsection.
\subsection{Blow-up estimates in the general case}
In this subsection, we use energy and ODE type estimates from previous sections and give the proof of Theorem \ref{th}.
\begin{proof}[Proof of Theorem \ref{th}]
  $(i)$ Let $R>0$, $a\in\mathbb{R}$ such that $(a,T(a))\in D_R$ and $(x,t)\in\mathcal{C}_{a,T(a),1}$. Consider $(\xi,\tau)$ the closest point of $\mathcal{C}_{a,T(a),1}$ to $(x,t)$. This means that 
$$ ||(x,t)-(\xi,\tau)||=\inf_{x'\in \mathbb{R}}\{||(x,t)-(x',|x'|=1-t)||\}=d((x,t),\mathcal{C}_{a,T(a),1}).$$
 By a simple geometrical construction, we see that it satisfies the following:
\begin{equation} \left\{
\begin{array}{l}
\tau=T(a)-(\xi-a) \\
\tau=t+(\xi-x),
\end{array}
\right . \label{ss}
\end{equation}
 hence, $T(a)-t-2 \xi+a+x=0$, so
\begin{align}\label{sss}
 \xi-x=\frac{1}{2} \left( (T(a)-t)+(a-x)\right).
\end{align}
Using the second equation of (\ref{ss}) and (\ref{sss}) we see that:
\begin{align*}
||(x,t)-(\xi,\tau)||=\sqrt{(\xi-x)^2+(\tau-t)^2}=\sqrt{2}|\xi-x|=\frac{\sqrt{2}}{2}| (T(a)-t)+(a-x)|.
\end{align*}
Thus,
\begin{align}\label{ssss} d((x,t),\Gamma)\ge d((x,t),\mathcal C_{a,T(a),1})=\frac{\sqrt{2}}{2}| (T(a)-t)+(a-x)|.
\end{align}
From Proposition \ref{p}, (\ref{ssss}) and the similarity transformation (\ref{trans_auto}) we have:
\begin{align*}
 e^{w_a(y,s)}&\le (T(a)-t)^2 e^{u(x,t)}\le C \frac{(T(a)-t)^2}{d((x,t),\Gamma)^2}\le C \left(\frac{T(a)-t}{T(a)-t-|x-a|}\right)^2 \\ &\le  \frac{C}{(1-|y|)^2}.
\end{align*}
Which gives the first inequality of $(i)$. The second one is given by Proposition \ref{p} and Lemma \ref{co}.\\
\medskip

$(ii)$ This is a direct consequence of Proposition \ref{ll} and Lemma \ref{co}.
\medskip

$(iii)$ Now, we will use section \ref{sec2} to prove this. Arguing by contradiction, we assume that $u$ is not global and that\\
$\forall  \epsilon_0>0,$ $\exists x_0 \in \mathbb{R}, \exists t_0 \in [0, T(x_0)),$ such that 

\begin{eqnarray*}
\frac{1}{T(x_0)-t_0} \displaystyle \int_{I} ((u_t(x,t_0))^2+(u_x(x,t_0))^2+e^{u(x,t_0)})dx < \frac{\epsilon_0}{(T(x_0)-t_0)^2},\\
\mbox{where }I=(x_0-(T(x_0)-t_0),x_0+(T(x_0)-t_0))
\end{eqnarray*}
We introduce the following change of variables:
\begin{equation*}
v(\xi, \tau)=u(x,t)+\log(T(x_0)-t_0),\mbox{ with }x=x_0+\xi(T(x_0)-t_0),\, t=t_0+\tau(T(x_0)-t_0).
\end{equation*}
Note that $v$ satisfies equation (\ref{waveq}). For $\epsilon_0=\bar \epsilon_0$, $v$ satisfies ($\ref{soleil}$), so, by Corollary \ref{pro}, $ v$ doesn't blow-up in $\{(\xi,\tau) |\, |\xi|< 1-\tau, \tau\in [0,1) \},$ thus, 
 $u$ doesn't blow-up in $\{(x,t) |\, |x-x_0|< T(x_0)-t, t\in [t_0,T(x_0)) \}$, which is a contradiction. This concludes the proof of Theorem \ref{th}.\\
\end{proof}
\subsection{Blow-up estimates in the non-characteristic case}
In this subsection, we prove Theorem \ref{NEW}. We give first the following corollary of Proposition \ref{p}.
\begin{cor} \label{cou}
Assume that $(u_0,u_1)\in W^{1,\infty}\times L^{\infty}$. Then, for all $R>0$, $a\in \mathcal{R}$ such that $(a,T(a))\in D_R$, we have for all $s\ge -\log T(a)$ and $|y|<1$
 \begin{eqnarray*}
| w_a(y,s)|\le C(R).
 \end{eqnarray*}
\end{cor}
\begin{proof}[Proof of Corollary \ref{cou}] Assume that $(u_0,u_1)\in W^{1,\infty}\times L^{\infty}$ and consider $R>0$ and $a\in \mathcal{R}$ such that $(a,T(a))\in D_R$.
 On the one hand, we recall from Proposition \ref{p} that
$$\forall (x,t)\in \mathcal {C}_{a,T(a),1},\;\frac{1}{ C}\le e^ud((x,t),\Gamma)^2\le C.$$
Using (\ref{trans_auto}), we see that
 \begin{eqnarray*}
\frac{1}{C}\le(\frac{d((x,t),\Gamma)}{T(a)-t})^2 e^{w_a(y,s)}\le C, \mbox{ with }y=\frac{x-a}{T(a)-t}\mbox{ and }s=-\log(T(a)-t).
 \end{eqnarray*}
Since
 \begin{eqnarray*}
\frac{1}{C}\le\frac{d((x,t),\Gamma)}{T(x)-t}\le C,
 \end{eqnarray*}
From Lemmas \ref{co} and \ref{so}, this yields to the conclusion of Corollary \ref{cou}.
\end{proof}
Now, we give the proof of Theorem \ref{NEW}.
\begin{proof}[Proof of Theorem \ref{NEW}] Consider $R>0$ and $a\in \mathcal{R}$ such that $(a,T(a))\in D_R$. We note first that the  fact that $|w_a(y,s)|\le C$ for all $|y|<1$ and $s\ge-\log T(a)$ follows from the Corollary \ref{cou}. It remains only to show that $\int_{-1}^{1}
\left( ({\partial_s w_a}(y,s))^2 + ({\partial_y w_a}(y,s))^2\right)\, dy\le C$. From Proposition \ref{p}, Lemmas \ref{co} and \ref{so}, we have:
\begin{equation}\label{l}
 \forall (x,t)\in \mathcal {C}_{a,T(a),1},\;  e^{u(x,t)}\le \frac{C(R)}{(T(a)-t)^2}.
\end{equation}
We define $ \mathcal{E}_a$ the energy of equation (\ref{waveq}) by

\begin{equation}\label{f}
 \mathcal{E}_a(t)=\frac{1}{2}\int_{|x-a|<T(a)-t}(\partial_t u(x,t))^2+(\partial_x u(x,t))^2dx-\int_{|x-a|<T(a)-t} e^{u(x,t)} dx.
\end{equation}
From Shatah-Struwe \cite{MR1674843}, we have
\begin{equation*}
\frac{d}{dt} \mathcal{E}_a(t)\le C e^{u(a-(T(a)-t),t)}+ C e^{u(a+(T(a)-t),t)}.
\end{equation*}
Integrating it over $[0,t)$ and using (\ref{l}), we see that
\begin{align*}
 \mathcal{E}_a(t)&\le \mathcal{E}_a(0)+C\int_0^t e^{u(a-(T(a)-s),s)}ds+C\int_0^t e^{u(a+(T(a)-s),s)}ds\\
&\le C(a,||(u_0,u_1)||_{H^{1}\times L^2 (-R,R)})+C\int_0^t \frac{ds}{(T(a)-s)^2}\le C+\frac{C}{(T(a)-t)}.
\end{align*}
Thus,
\begin{equation}\label{k}
 \mathcal{E}_a(t) \le\frac{C}{(T(a)-t)}.
\end{equation}
Now, using (\ref{l}) and (\ref{k}) to bound the two first terms in the definition of $\mathcal{E}_a(t)$ (\ref{f}), we get
\begin{align}\label{lo}
\notag \frac{1}{2}\int_{|x-a|<T(a)-t}&(\partial_t u(x,t))^2+(\partial_x u(x,t))^2dx \le \int_{|x-a|<T(a)-t} e^u dx+\frac{C}{(T(a)-t)}\\ &\le \int_{|x-a|<T(a)-t} \frac{C}{(T(a)-t)^2}dx+ \frac{C}{(T(a)-t)}\le \frac{C}{(T(a)-t)}.
\end{align}

Writing inequality (\ref{lo}) in similarity variables, we get for all $s\ge -\log T(a)$,
$$\int_{-1}^1(\partial_s w_a(y,s))^2+\int_{-1}^1(\partial_y w_a(y,s))^2\le C(R).$$
This yields to the conclusion of Theorem \ref{NEW}.
\end{proof}
Now, we give the proof of Proposition \ref{2.1}.
\begin{proof}[Proof of Proposition \ref{2.1}]
 
 Multiplying (\ref{equaw}) by $\partial_s w$, and integrating
over (-1,1), we see that
\begin{eqnarray*}
 \int_{-1}^{1} {\partial_s w}{\partial^2_{s} w} \, dy-\int_{-1}^{1} {\partial_s w}{\partial_y} ((1-y^2){\partial_{y} w})\,dy-\int_{-1}^{1} {\partial_s w}e^w\,dy\\+2\int_{-1}^{1} {\partial_s w}\,dy=-\int_{-1}^{1}  ({\partial_s w})^2\,dy+2  \int_{-1}^{1} y{\partial_s w} \partial^2_{y,s}
 w\, dy.
\end{eqnarray*}
Thus,
\begin{eqnarray*}
\frac{d}{ds} \left(  \int_{-1}^{1} \frac{1}{2}({\partial_s w})^2+ 2
w -e^w \, dy \right) +I_1 = -\int_{-1}^{1}({\partial_s w})^2\,
dy+I_2,
\end{eqnarray*}
where,
\begin{eqnarray*}
 I_1&=&-\int_{-1}^{1} \partial_y \left( (1-y^2)
\partial_y w \right) \partial_s w \, dy
=\int_{-1}^{1} \left( (1-y^2)\partial_y w \right)
\partial_{ys}^2 w  \, dy \\&=&\frac{1}{2} \frac{d}{ds}\left( \int_{-1}^{1}(1-y^2)({\partial_s w})^2 \, dy\right)
\end{eqnarray*}and
\begin{eqnarray*}
 I_2&=&-2  \int_{-1}^{1} y{\partial_s w} \partial^2_{y,s} w\,dy=-\int_{-1}^{1} y \partial_y (\partial_s w)^2\,dy\\&=&\int_{-1}^{1}(\partial_s w)^2\,dy-({\partial_s w(-1,s)})^2-(
{\partial_s w(1,s)})^2.
\end{eqnarray*}
Thus,
\begin{eqnarray*}
\frac{d}{ds} \left (\int_{-1}^{1} \left[\frac{1}{2}
 ({\partial_s
w})^2 +\frac{1}{2} (1-y^2)( {\partial_y w} )^2-e^w+2 w\right]\, dy\right
)=-({\partial_s w(-1,s)})^2-( {\partial_s w(1,s)})^2
\end{eqnarray*}\\
which yields the conclusion of Proposition \ref{2.1} by integration in
time.
\end{proof}

\medskip
\noindent{\bf Address:}\\
Universit\'e de Cergy-Pontoise, 
Laboratoire Analyse G\'eometrie Mod\'elisation, \\
  CNRS-UMR 8088, 2 avenue Adolphe Chauvin
95302, Cergy-Pontoise, France.
\\ \texttt{e-mail: asma.azaiez@u-cergy.fr}\\
Courant Institute, NYU, 251 Mercer Street, NY 10012, New York.
\\ \texttt{e-mail: masmoudi@cims.nyu.edu}\\
Universit\'e Paris 13, Institut Galil\'ee, Laboratoire Analyse G\'eometrie et Applications, \\
  CNRS-UMR 7539, 99 avenue J.B. Cl\'ement 93430, Villetaneuse, France.
\\ \texttt{e-mail: zaag@math.univ-paris13.fr}
\end{document}